\documentclass[reqno,12pt]{amsart}
\usepackage{
    enumitem,
}
\usepackage{amsaddr}
\usepackage{color}

\usepackage[margin=2.2cm]{geometry}

\newtheorem{theorem}{Theorem}[section]
\newtheorem{lemma}[theorem]{Lemma}
\newtheorem{prop}[theorem]{Proposition}

\theoremstyle{definition}

\theoremstyle{remark}
\newtheorem{remark}[theorem]{Remark}

\numberwithin{equation}{section}

\newcommand{\R}{\mathbb{R}}

\DeclareMathOperator\supp{supp}

\begin{document}

\title[Anisotropic and isotropic persistent
singularities]{Anisotropic and isotropic persistent
singularities of solutions of the fast diffusion equation}

\author[M. Fila]{Marek Fila}
\address{Department of Applied Mathematics and Statistics, 
Comenius University, \\ 
842 48 Bratislava, Slovakia}
\email{fila@fmph.uniba.sk}
\thanks{
The first author was partially supported by the Slovak Research and Development Agency under the contract No. APVV-18-0308 and by  VEGA grant 1/0339/21. 
The second author was partially supported by
VEGA grant 1/0339/21 and by Comenius University grant UK/111/2021. 
The third author was partially supported by 
JSPS KAKENHI Early-Career Scientists (No.~19K14567). 
The fourth author was partially supported by 
JSPS  KAKENHI Grant-in-Aid for Scientific Research (A) (No.~17H01095).
}

\author[P. Mackov\'a]{Petra Mackov\'a}
\address{Department of Applied Mathematics and Statistics, 
Comenius University, \\ 
842 48 Bratislava, Slovakia}
\email{petra.mackova@fmph.uniba.sk}

\author[J. Takahashi]{Jin Takahashi}
\address{Department of Mathematical and Computing Science, 
Tokyo Institute of Technology, \\ 
Tokyo 152-8552, Japan}
\email[Communicating author]{takahashi@c.titech.ac.jp}

\author[E. Yanagida]{Eiji Yanagida}
\address{Department of Mathematics, Tokyo Institute of Technology, \\
Tokyo 152-8551, Japan}
\email{y-aska@msc.biglobe.ne.jp}



\begin{abstract}
The aim of this paper is to study a class of positive solutions of the fast diffusion equation with specific persistent singular behavior. 
First, we construct new types of solutions with anisotropic singularities. Depending on parameters, either these solutions solve the original equation in the distributional sense, or they are not locally integrable in space-time.
We show that the latter also holds for solutions with snaking singularities, whose existence has been proved recently by M. Fila, J.R. King, J. Takahashi, and E. Yanagida.
Moreover, we establish that in the distributional sense, 
isotropic solutions whose existence was proved by M. Fila, J. Takahashi, and E. Yanagida in 2019, actually solve the corresponding problem with a moving Dirac source term. 
Last, we discuss the existence of solutions with anisotropic singularities in a critical case. 
\end{abstract}

\maketitle

\section{Introduction}
Let $n \geq 2$ and $m \in (0,1)$. We study positive singular solutions of the fast diffusion equation
\begin{equation} \label{eq:main}
    u_t = \Delta u^m, \qquad x \in \R^n \setminus \{ \xi_0 \}, \quad t>0,
\end{equation}
with an initial condition
\begin{equation} \label{eq:main_initial}
    u(x,0)=u_0(x), \qquad x \in \R^n \setminus \{\xi_0\}.
\end{equation}
Here, $\xi_0 \in \R^n$ is a given point at which solutions are singular, i.e.
\begin{equation*}
    u(x,t) \to \infty \quad \mbox{ as } \quad x \to \xi_0, \quad t>0.
\end{equation*}
Let $S^{n-1} := \{ x \in \R^n: |x| = 1 \}$ denote the unit $\displaystyle (n-1)$-sphere and set 
\begin{equation} \label{eq:r_omega}
    r := |x-\xi_0| \quad \text{ and } \quad \omega := (x-\xi_0)/|x-\xi_0|. 
\end{equation}
Let $\lambda>0$ and $\alpha\in C^{2,1}(S^{n-1} \times [0,\infty))$ be positive.
The aim of this paper is to study positive solutions with the 
persistent singular behavior of the form
\begin{equation} \label{eq:anisotrop_asymp}
    u(x,t) = \alpha(\omega,t) r^{-\lambda} + o(r^{-\lambda} )\quad \text{ as } \quad r \to 0,
\end{equation}
for $\omega \in S^{n-1}$ and $t\geq 0$.
We say that if $\alpha(\omega,t)$ depends non-trivially on the space variable $\omega$, the corresponding solution $u$ has an anisotropic singularity, otherwise it is asymptotically radially symmetric.

Our main result formulated in Theorem~\ref{th:anisotropic} concerns the existence of solutions of~\eqref{eq:main}-\eqref{eq:main_initial} with anisotropic singularities.
In order to prove the existence of such solutions, we introduce the following assumptions.
\begin{itemize}
\item[(A1)]
    Let $\alpha\in C^2(S^{n-1})$ be positive.
    
\item[(A2)]
    Let $0<m<1$ and let $\lambda$, $\nu$ satisfy 
    \[
    	\lambda>\frac{2}{1-m}, \qquad 
    	(1-m)\lambda-2-m(\lambda-\nu)>0, \qquad 
    	\lambda> \nu>0.
    \]
\item[(A3)]
    Let $u_0\in C(\R^n\setminus\{\xi_0\})$ be positive and such that it has the asymptotic behavior
    \begin{equation*}
        u_0^m(x)=\alpha^m(\omega) |x-\xi_0|^{-m\lambda} + O\big(|x-\xi_0|^{-m\nu}\big) \quad \text{ as } \quad x\to \xi_0,
    \end{equation*}
    for each $\omega \in S^{n-1}$, and
    \begin{equation*}
       C^{-1}\leq  u_0(x) \leq C
    \end{equation*}
    for $|x-\xi_0|\geq 1$ and some constant $C>1$.
\end{itemize}

Note that the condition (A2) implies that $\nu$ is sufficiently close to $\lambda$.

\begin{theorem} \label{th:anisotropic}
Let $n \geq 2$ and assume \emph{(A1)}, \emph{(A2)}, and \emph{(A3)}.
Then there is a function 
\[
    u\in C^{2,1}( \{(x,t)\in (\R^n\setminus\{\xi_0\})\times(0,\infty)\}) \cap C( \{(x,t)\in (\R^n\setminus\{\xi_0\})\times[0,\infty)\}),
\]
which satisfies~\eqref{eq:main}-\eqref{eq:main_initial} pointwise and
\begin{equation*}\label{eq:sing1}
    u(x,t)^m = \alpha^m(\omega) |x-\xi_0|^{-m\lambda} + O(|x-\xi_0|^{-m\nu}) \quad \text{ as } \quad x\to \xi_0,
\end{equation*}
for each $\omega \in S^{n-1}$ and $t\geq 0$. 
\end{theorem}

A subclass of solutions from Theorem~\ref{th:anisotropic} has been also studied in~\cite{TY2021}. The authors of~\cite{TY2021} focused on radially symmetric solutions of~\eqref{eq:main}-\eqref{eq:main_initial} with $n\geq 3$, $0<m<m_c:=(n-2)/n$ and with the initial condition $u_0(x)=(c_1^m |x-\xi_0|^{-m\lambda} + c_2^m)^{1/m}$, where $2/(1-m)<\lambda<(n-2)/m$, $c_1>0$, and $c_2 \geq 0$. In addition to the existence, several interesting properties of these solutions have been proved, among them their uniqueness.

In our next result we show that, depending on parameters, solutions constructed in Theorem~\ref{th:anisotropic} either solve the original fast diffusion equation in the distributional sense, i.e.
\begin{equation} \label{eq:weak}
    u_t=\Delta u^m \quad \text{ in } \quad \mathcal{D}'(\R^n \times (0, \infty)),
\end{equation} 
or they are not locally integrable in space-time. 

\begin{theorem} \label{th:anisotropic_weak}
Let the conditions from Theorem~\ref{th:anisotropic} be satisfied.
\begin{enumerate}[label={\upshape(\roman*)}, align=left, widest=iii, leftmargin=*]
\item \label{it:i} 
    If $\lambda<n$, $0 < m < m_c$, and $n > 2$, then for the solution $u$ from Theorem~\ref{th:anisotropic} it holds that $u \in C([0,\infty);L^1_{loc}(\R^n))$, and it satisfies~\eqref{eq:weak} in the distributional sense, i.e.
	\[
        \int_0^\infty \int_{\R^n} \big( u \varphi_t + u^m \Delta\varphi \big) \, dy \, dt = 0
    \]
    for all $\varphi \in C_0^\infty (\R^n \times (0, \infty))$.
    
\item \label{it:ii} 
    If $\lambda \geq n$, then the solution $u$ from Theorem~\ref{th:anisotropic} satisfies $u \notin L^p_{loc}(\R^n \times [0,\infty))$ for any $p \geq 1$.
\end{enumerate}
\end{theorem}

We note that in the supercritical exponent range $m_c < m < 1$, the authors of \cite{HP} proved that all solutions of~\eqref{eq:weak} with $u_0 \in L^1_{loc}(\R^n)$ become locally bounded and continuous for all $t>0$. 

A further related result concerning anisotropic singularities can be found in~\cite{FKTY}. Here, the authors constructed positive entire-in-time solutions with snaking singularities for the fast diffusion equation (in the range $m^*<m<1$ and $n \geq 2$, where $m^*:=(n-3)/(n-1)$ when $n \geq 3$ and $m^*:=0$ when $n=2$). In particular, these solutions have a singularity on a set $\Gamma(t):=\{\xi(s); -\infty<s<ct\}$ for $c>0$ and each $t \in \R$. Here $\xi:\R\rightarrow \R^n$ satisfies 
Condition 1.1 in~\cite{FKTY}.
Their construction was based on the existence of the following explicit singular traveling wave solution with cylindrical symmetry
\begin{equation} \label{eq:snaking}
   U(x,t) = C \big( |a| |x-ta| + a \cdot (x-ta) \big)^{-\frac{1}{1-m}},
\end{equation}
where $a$ is a velocity vector, and $C$ is an explicitly computable constant. As in Theorem~\ref{th:anisotropic_weak}~\ref{it:ii}, the solution $U$ is also an example of a function with no local integrability in space-time. Namely, in Section~\ref{subsec:proof_snaking} we show the following.
\begin{remark} \label{rem:snaking}
    Let $n \geq 2$ and $m^*<m<1$. Then for the function $U$ from~\eqref{eq:snaking} it holds that
    $U \notin L^p_{loc}(\R^n \times \R)$ for any $p \geq 1$. 
\end{remark}

To extend the idea of various possibilities of distributional solutions of the fast diffusion and porous medium equation, we present our last result in Theorem~\ref{th:dirac}. Here, a class of asymptotically radially symmetric singular solutions satisfies the corresponding equation with a moving Dirac source term in the distributional sense.
The existence of such solutions of the initial value problem
\begin{align}
     \label{eq:porous}
     u_t &= \Delta u^m, \qquad \,\, x \in \R^n \setminus \{ \xi(t) \}, \quad t \in (0, \infty), \\
     \label{eq:porous-init}
     u(x,0)&=u_0(x), \qquad x \in \R^n \setminus \{\xi(0)\},
\end{align}
was established in Theorem~1.1 in~\cite{FTY}.
Assuming $n \geq 3$ and $m>m_*:= (n-2)/(n-1)$, the authors of~\cite{FTY} constructed singular solutions of~\eqref{eq:porous}-\eqref{eq:porous-init}, which for some given $C^1$ function $k(t)$ behave as 
\begin{equation} \label{eq:rad_as_sym}
    u^m(x,t) = k^m(t) |x-\xi(t)|^{-(n-2)} + o(|x-\xi(t)|^{-(n-2)})\quad \text{ as } \quad x \to \xi(t).
\end{equation}
It was shown in~\cite{FTY} that $m_*$ is a critical exponent for the existence of such solutions, and there are no such solutions if $m<m_*$.
To construct global-in-time solutions of this form, suitable conditions on $\xi'$, $k$, and $k'$ were imposed. 

\begin{remark} \label{rem:1}
    The existence of solutions from~\cite{FTY} can be extended to the parameter range $n\geq 3$ and $m_c < m \leq m_*$ if the singularity is not moving, i.e. if $\xi(t) \equiv \xi_0$. 
    This can be verified by an inspection of the proof of Theorem~1.1 in~\cite{FTY}, which is in this case simpler since all terms containing $\xi'$ vanish.
\end{remark}

We also remark that the results from~\cite{FTY} have been extended previously in a different way in~\cite{FMTY}. Here, the authors treated the case $n=2$, $m > m_*=0$. They established the existence of solutions that, near the singularity, behave like the fundamental solution of the Laplace equation to the power $1/m$. 

In Theorem~\ref{th:dirac} we show that solutions from~\cite{FTY} satisfy
\begin{equation} \label{eq:n3_weak_delta}
    u_t = \Delta u^m + (n-2) |S^{n-1}| k^m(t) \delta_{\xi(t)}(x) \quad \text{ in } \quad \mathcal{D}'(\R^n \times (0, \infty)).
\end{equation}
Here, $\delta_{\xi(t)}$ denotes the Dirac measure on $\R^n$, giving unit mass to the point $\xi(t) \in \R^n$. 

\begin{theorem} \label{th:dirac}
    Let $n\geq3$ and assume that conditions on $k(t)$ and $u_0(x)$ from Theorem~1.1 in~\cite{FTY} hold. Let either $m > m_*$ and $\xi(t)$ be as in Theorem~1.1 in~\cite{FTY}, or $m_c < m \leq m_*$ and $\xi(t) \equiv \xi_0$.
    Then the solution $u$ satisfies equation~\eqref{eq:n3_weak_delta} in the distributional sense, i.e.
\[
        -\int_0^\infty \int_{\R^n} \big( u \varphi_t + u^m \Delta\varphi \big) \, dx \, dt = \int_0^\infty (n-2) |S^{n-1}| k^m(t) \varphi(\xi(t),t) \, dt
\]
    for all $\varphi \in C_0^\infty (\R^n \times (0, \infty))$.
\end{theorem}

Equations~\eqref{eq:main} and~\eqref{eq:porous} with $n\geq3$ have radially symmetric stationary solutions of the form
\begin{equation} \label{eq:steady_state}
	\tilde u(x) = K|x-\xi_0|^{-(n-2)/m}, \qquad x \in \R^n \setminus \{ \xi_0 \},
\end{equation}
where $K$ is an arbitrary positive constant,
and these solutions satisfy
\begin{equation} \label{eq:fund_weak}
    -\Delta \tilde u(x) = (n-2)|S^{n-1}| K \delta_{\xi_0}(x) \quad \text{ in } \quad \mathcal{D}'(\R^n),
\end{equation}
where $|S^{n-1}|$ is the hypervolume of the $(n-1)$-dimensional unit sphere. Hence, the result of Theorem~\ref{th:dirac} can be expected.
In~\cite{KT14}, the authors constructed singular solutions with time-dependent singularities for the heat equation
\begin{equation*} 
    u_t = \Delta u + w(t) \delta_{\xi(t)}(x) \quad \text{ in } \quad \mathcal{D}'(\R^n \times (0, T)),
\end{equation*}
where $n \geq 2$, $T \in (0,\infty]$, and $w \in  L^1((0, t))$ for each $t \in (0,T)$. 
The behavior of solutions from~\cite{KT14} near the singularity does not always have to be like that of the fundamental solution of the Laplace equation, and the profile loses the asymptotic radial symmetry.
Further results concerning the heat equation
\begin{equation*} 
	u_t = \Delta u + \delta_{\xi(t)}(x) \otimes M(t) \quad \text{ in } 
	\quad \mathcal{D}'(\R^n \times (0, T)),
\end{equation*}
where $\delta_{\xi(t)}(y) \otimes M(t)$ is a product measure of $M(t)$ and $\delta_{\xi(t)}(x)$, can be found in~\cite{KT16}. 
Solutions of the porous medium ($m>1$) and fast diffusion equation in the supercritical range ($m_c < m < 1$) with singularities which are not necessarily standing were analyzed in \cite{Lu1} and \cite{Lu2}, respectively. 
If $M$ is a nonnegative Radon measure on $\R^{n+1}$, which satisfies $M(\Omega \times (0,T)) < \infty$ for $T>0$ and a bounded domain $\Omega \subset \R^n$, then there exists a function $u$ such that
\begin{equation*} 
	u_t = \Delta u^m + M(x, t) \quad \text{ in } \quad \mathcal{D}'(\Omega \times (0, T)),
\end{equation*}
and $u^m \in L^q((0, T); W^{1,q}_0 (\Omega) )$ with $1 < q < 1 + 1/(1 + mn)$.

A moving Dirac measure on the right-hand side of parabolic systems also appears in several biological applications concerning, for example, the growth of axons or angiogenesis. See~\cite{CZ} and \cite{Bookholt}, respectively. 
A moving Dirac measure also appears in~\cite{PO}, where the authors studied the Cattaneo telegraph equation with a moving time-harmonic source in the context of the Doppler effect.

We also mention the following two results, which can be applied to solutions with anisotropic singularities.
When $n \geq 3$, $0 < m < m_c$, and the singularity of the initial function satisfies $a_1 |x-\xi_0|^{-2/(1-m)} \leq u_0(x) \leq a_2 |x-\xi_0|^{-2/(1-m)}$ for some $a_1, a_2 > 0$ and for all $x \in \Omega$, then from \cite{vazquez2} (for $\Omega = \R^n$) and \cite{VW} (for smoothly bounded domain $\Omega \subset \R^n$) it results that finite-time blow-down occurs. More specifically, there is a $T > 0$ such that $u(\cdot, t) \notin L^\infty (\Omega)$ for $t < T$ but $u(\cdot, t) \in L^\infty (\Omega)$ for $t > T$, i.e. that the singularity disappears after a time $T$.
On the other hand, if $m$ is in the range $m_c < m < 1$, the authors of~\cite{ChV} concluded the monotonicity of strongly singular sets of extended solutions, i.e. that it cannot shrink in time. Hence, the singularity of such solutions persists for all times.

This paper is organized as follows.
A formal analysis of solutions with the asymptotic behavior~\eqref{eq:anisotrop_asymp} is given in Section~\ref{subses:fc}. 
The last part of this section is devoted to a critical case that is left as an open problem.
The existence result in Theorem~\ref{th:anisotropic} is then proved in Section~\ref{subsec:proof_thn}. 
Formal computations in Section~\ref{subses:fc} suggest the choice of comparison functions in Subsections~\ref{subsec:super} and~\ref{subsec:sub}.
We leave the question of extending the results from Theorem~\ref{th:anisotropic} from standing to moving singularities open. We see no problem in using the methods employed in this text, however, different critical exponents and technical difficulties may arise. 
We continue with the proof of Theorem~\ref{th:dirac} in Section~\ref{subsec:proof_th3}, and the proof of Theorem~\ref{th:anisotropic_weak} in Section~\ref{subsec:proof_anisotropic_weak}. Finally, computations concerning Remark~\ref{rem:snaking} are given in Section~\ref{subsec:proof_snaking}.

\section{Formal computations} 
\label{subses:fc}

Let $u(x,t)$ be given by~\eqref{eq:anisotrop_asymp}, i.e. 
\begin{equation*} 
    u(x,t) = \alpha(\omega,t) r^{-\lambda} + o(r^{-\lambda}),
\end{equation*}
where $r>0$, $\omega \in S^{n-1}$, $t\geq 0$, and recall notation~\eqref{eq:r_omega} for $r$ and $\omega$.
For $u(x,t) = w(r,\omega,t)$, the fast diffusion equation~\eqref{eq:main} is transformed into
\begin{equation} \label{eq:wprob}
    w_t = r^{1-n} \frac{\partial}{\partial r} \left(r^{n-1} \frac{\partial w^m}{\partial r} \right) + \frac{1}{r^2} \Delta_\omega w^m.
\end{equation}
Here, $\Delta_\omega$ denotes the Laplace-Beltrami operator on $S^{n-1}$.
Simple computations show that
\begin{equation} \label{eq:formal_comp}
\begin{split}
    w_t &= \alpha_t r^{-\lambda} + o(r^{-\lambda}), \\
    \Delta w^m &= \big(\Delta_\omega \alpha^m -m\lambda(n-2-m\lambda)\alpha^m \big) r^{-m\lambda -2} + o(r^{-m\lambda -2}).
\end{split}    
\end{equation}
The leading term is different in each of the three cases: $\lambda>m\lambda + 2$, $\lambda<m\lambda + 2$, and $\lambda=m\lambda + 2$.

\subsection{$\boldsymbol{\lambda> 2/(1-m)}$} 
\label{sec:results2}
The most singular case is $\lambda>m\lambda + 2$, which is equivalent to $\lambda> 2/(1-m)$. This implies that the leading term in~\eqref{eq:formal_comp} is $w_t$, hence, we set $\alpha = \alpha(\omega)$. 
The existence result in Theorem~\ref{th:anisotropic} is based on this observation.

\subsection{$\boldsymbol{\lambda< 2/(1-m)}$} 
\label{sec:results1}
The case $\lambda<m\lambda + 2$ is equivalent to $\lambda< 2/(1-m)$. The leading term in~\eqref{eq:formal_comp} is $\Delta w^m$, which implies that $\alpha$ must be a solution of
\begin{center}
    $-\Delta_\omega \alpha^m  = -m\lambda(n-2-m\lambda)\alpha^m$.
\end{center}
Eigenvalues of $-\Delta_\omega$ are non-negative and start with zero (the constant $1$ is the corresponding eigenfunction), other eigenfunctions change sign, see~\cite{LaplaceBeltrami}.
Since we are looking for positive solutions, we obtain conditions
    $$\lambda = \frac{n-2}{m}, \quad m > m_c, \quad \text{and} \quad n \geq 3 .$$
As we pointed out in the introduction, the existence of the corresponding asymptotically radially symmetric solutions for $n\geq 3$ and $m > m_*$ in the case of a moving singularity was established in Theorem~1.1 in~\cite{FTY}. 
Moreover, in Remark~\ref{rem:1} we explain that in the case of a standing singularity, the proof of Theorem~1.1 in~\cite{FTY} is valid also in the parameter range $n\geq 3$, $m_c < m \leq m_*$.
Our result extending the qualitative analysis of these solutions can be found in Theorem~\ref{th:dirac}.

\subsection{Critical case $\boldsymbol{\lambda = 2/(1-m)}$, open problem} 
\label{sec:results3}
In the critical case $\lambda = 2/(1-m)$, the terms $w_t$ and $\Delta w^m$ are balanced. 
Let 
\[
    A:=m\lambda(m\lambda-n+2) = \frac{2 m n (m-m_c)}{(1 - m)^2}. 
\]
Balancing the leading terms in~\eqref{eq:formal_comp} leads us to an initial value problem
\begin{align}
     \label{eq:alfa_eq}
     \alpha_t(\omega,t) &= \Delta_\omega \alpha^m(\omega,t) + A\alpha^m(\omega,t), \qquad \omega \in S^{n-1}, \quad 0 <t <T, \\
     \label{eq:alfa_eq_initial}
     \alpha(\omega,0) &= \alpha_0(\omega)>0, \qquad \qquad \qquad \qquad \, \omega \in S^{n-1},
\end{align}
where $T \in (0,\infty]$.
If we prove the existence of a positive classical solution $\alpha$ of~\eqref{eq:alfa_eq}-\eqref{eq:alfa_eq_initial} for some $T>0$, we obtain a positive classical solution of~\eqref{eq:main}-\eqref{eq:main_initial} of the form
\[
    u(x,t) = \alpha(\omega,t) |x-\xi_0|^{-\lambda}, \qquad x \in \R^n \setminus \{ \xi_0 \}, \quad 0 <t <T.
\]

At the end of this section, we present some examples of solutions of~\eqref{eq:alfa_eq}-\eqref{eq:alfa_eq_initial}.
A well-known explicit solution is
\[
    \tilde \alpha (t) = \big((1-m)A \, t + t_0 \big)^{\frac{1}{1-m}},
\]
where $t_0$ is an arbitrary positive constant.
In order to obtain solutions of~\eqref{eq:main}-\eqref{eq:main_initial} with an anisotropic singularity, we are interested in solutions of~\eqref{eq:alfa_eq}-\eqref{eq:alfa_eq_initial} that, unlike $\tilde \alpha$, depend non-trivially on the space variable $\omega$.
Such solutions can be obtained by looking for solutions of the form $\alpha (\omega,t) = \tau(t) \beta^{1/m}(\omega)$, where $\beta$ is non-constant. 
Using the method of separation of variables, we have $\tau(t) = ((1-m) C t + t_0)^{\frac{1}{1-m}}$, where $t_0>0$ and $C$ is a constant from the separation of variables. For $\beta(\omega)$ we obtain a semilinear elliptic equation on a sphere
\begin{equation} \label{eq:onsphere}
    \Delta_\omega \beta(\omega) + A \beta(\omega) = C \beta^{1/m}(\omega), \qquad \omega \in S^{n-1}.
\end{equation}
We briefly examine the existence of a class of solutions of~\eqref{eq:onsphere} depending only on an angle $\theta \in [0,2\pi)$. In this case, equation~\eqref{eq:onsphere} becomes $\ddot \beta(\theta) + A \beta(\theta) = C \beta^{1/m}(\theta)$.
It represents a Hamiltonian system
\[
    \begin{cases}
      \dot \beta = v, \\
      \dot v = - \beta (A - C \beta^{{1/m}-1}),
    \end{cases}
\]
with a relevant critical point $P=((A/C)^{\frac{m}{1-m}},0)$ if $C \neq 0$ has the same sign as $A \neq 0$. Notice that this condition guarantees the same asymptotic behavior of $\tau$ as that of $\tilde \alpha$, 
which is consistent with the results from~\cite{vazquez2, ChV} described in the introduction. Finally, the existence of periodic trajectories results from a standard ODE theory: the critical point $P$ is a center, i.e. all trajectories close to it are closed orbits if $A<0$. 

The existence of a more general class of classical positive solutions of~\eqref{eq:alfa_eq}-\eqref{eq:alfa_eq_initial}, which depend non-trivially on $\omega$, is left as an open problem.

\section{Proof of Theorem~\ref{th:anisotropic}} \label{subsec:proof_thn}

\subsection{Construction of supersolutions} \label{subsec:super}
We set $a(t) := A e^{At}$, where 
$A \geq 1$ is a sufficiently large constant 
chosen later, and define a function
\begin{equation} \label{eq:supersolform_n}
    w^+(r,\omega,t) := \big( \alpha^m(\omega) r^{-m\lambda} + a(t) r^{-m\nu} +A\big)^\frac{1}{m}.
\end{equation}
In what follows, we prove that $w^+$ is a supersolution of~\eqref{eq:wprob}. 

\begin{lemma} \label{l+:n}
    Let $n \geq 2$ and assume \emph{(A1)} and \emph{(A2)}. Then there exists constant $A \geq 1$, such that the function $w^+(r,\omega,t)$ defined in~\eqref{eq:supersolform_n} is a supersolution of equation~\eqref{eq:wprob} for $r>0$, $\omega \in S^{n-1}$, 
    and $t>0$.
\end{lemma}

\begin{proof}
We define a bounded function
\[
    \sigma(\omega) := \Delta_\omega \alpha^m(\omega) + m\lambda (m\lambda-n+2) \alpha^m(\omega)
\]
and compute
\begin{equation*}
\begin{split}
    w^+_t &= \frac{1}{m} A a r^{-m\nu} \big(\alpha^m r^{-m\lambda} + a r^{-m\nu} +A\big)^{\frac{1}{m}-1}, \\
    -\Delta (w^+)^m &= -\sigma r^{-m\lambda-2} - m\nu (m\nu-n+2) a r^{-m\nu-2}.
\end{split}    
\end{equation*}
Since 
$(1-m)\lambda-2-m(\lambda-\nu)>0$,
$\alpha \geq \alpha_{min} :=\min_{\omega \in S^{n-1}} \alpha(\omega) >0$, $\sigma \leq \sigma_{max}:= \max_{\omega\in S^{n-1}} \sigma(\omega) < \infty$, 
and $A\geq1$, for $r\leq 1$ we obtain
\begin{equation*}
\begin{split}
     &w^+_t -\Delta (w^+)^m \\
     &= \frac{1}{m} A a \big(\alpha^m + a r^{m(\lambda-\nu)} + A r^{m\lambda} \big)^{\frac{1}{m}-1} r^{-\lambda + m(\lambda-\nu)} 
     - \sigma r^{-m\lambda-2} - m\nu (m\nu-n+2) a r^{-m\nu-2} \\
     &\geq \frac{1}{m} A a \alpha^{1-m}
     r^{-m\lambda-2} - \sigma r^{-m\lambda-2} - m\nu (m\nu+2) a r^{-m\nu-2} \\ 
     &\geq \left( \Big( \frac{1}{m} A \alpha_{min}^{1-m} - m\nu (m\nu+2) \Big) a - \sigma_{max} \right) r^{-m\lambda-2}.
\end{split}
\end{equation*}
Thus, for $A \geq m \alpha_{min}^{-(1-m)} (\sigma_{max} + m\nu (m\nu+2))$, it holds that $w^+_t -\Delta (w^+)^m \geq 0$ for all $r\leq 1$, $\omega \in S^{n-1}$, 
and $t>0$.
Similarly, for $r>1$, $\omega \in S^{n-1}$, and $t > 0$ we have 
\begin{equation*}
\begin{split}
     w^+_t -\Delta (w^+)^m 
     &\geq \frac{1}{m}  A^\frac{1}{m} a r^{-m\nu} - \sigma_{max} r^{-m\lambda-2} - m\nu (m\nu-n+2) a r^{-m\nu-2} \\
     &\geq \left( \Big(\frac{1}{m} A - m\nu (m\nu-n+2) \Big) a - \sigma_{max} \right) r^{-m\nu-2}.
\end{split}
\end{equation*}
This completes the proof that for any $A \geq 1$ sufficiently large, the function $w^+$ defined in~\eqref{eq:supersolform_n} is a supersolution of~\eqref{eq:wprob} for $t>0$ in the whole space.
\end{proof}

\subsection{Construction of subsolutions} \label{subsec:sub}
Let $\mu>\lambda$ satisfy
\begin{equation}\label{eq:asmu}
	m\mu (m\mu+2-n) \min_{\omega \in S^{n-1}} \alpha^m - \max_{\omega\in S^{n-1}}|\Delta_\omega(\alpha^m)| >0.
\end{equation}
Note that~\eqref{eq:asmu} implies $\mu>(n-2)/m$.
Let $\delta>0$ satisfy
\[
	0<\delta<\frac{\lambda-\nu}{\mu-\nu},
\] 
and define
\[
	b(t):=b_0 e^{Bt}, \qquad 
	\rho(t):= (1-\delta)^\frac{1}{m(\lambda-\nu)} 
	b^{-\frac{1}{m(\lambda-\nu)}}(t), 
\]
where $b_0, B>1$ are sufficiently large constants chosen later. 
We set 
\[
\begin{aligned}
	w_{in}^-(r,\omega,t) &:= \alpha(\omega) r^{-\lambda} 
	( 1-b(t) r^{m(\lambda-\nu)} )^\frac{1}{m}, \\
	w_{out}^-(r,\omega,t) &:= \alpha(\omega) \delta^\frac{1}{m} 
	\rho^{\mu-\lambda}(t) r^{-\mu}.
\end{aligned}
\]
Note that the zero point of $w_{in}^-$ is $b^{-\frac{1}{m(\lambda-\nu)}}(t)$
and that $w_{in}^-$ intersects $w_{out}^-$ at $r=\rho(t)<1$. 
Now we can construct a subsolution of the form 
\begin{equation} \label{eq:subsolform}
	w^-(r,\omega,t)= 
	\left\{ 
\begin{aligned} 
	& w_{in}^-(r,\omega,t) 
	&&\mbox{ for }r\leq \rho(t), \text{ } t\geq 0, \\
	& w_{out}^-(r,\omega,t)
	&&\mbox{ for }r> \rho(t), \text{ } t\geq 0. 
\end{aligned}
	\right.
\end{equation}

\begin{lemma} \label{l-:n}
Let $n \geq 2$ and assume \emph{(A1)} and \emph{(A2)}. Then there exist constants $b_0, B>1$, such that the function $w^-$ defined in~\eqref{eq:subsolform} is a subsolution of equation~\eqref{eq:wprob} for $r>0$, $\omega \in S^{n-1}$, and $t >0$.
\end{lemma}

\begin{proof}
\textbf{Inner part:}
Let $t>0$.
We consider the inner part $r\leq \rho(t)$. 
Straightforward computations show that 
\[
\begin{aligned}
	(w_{in}^-)_t &=
	\frac{1}{m} \alpha r^{-\lambda} 
	( 1-b r^{m(\lambda-\nu)} )^{\frac{1}{m}-1} 
	(-b' r^{m(\lambda-\nu)} ) \\
	&= 
	- \frac{1}{m} B b 
	\alpha r^{-\lambda+m(\lambda-\nu)} 
	( 1-b r^{m(\lambda-\nu)} )^{\frac{1}{m}-1}
\end{aligned}
\]
and 
\[
\begin{aligned}
	\Delta (w_{in}^-)^m =& \, m \alpha^m ( (m\lambda+2-n) \lambda r^{-m\lambda-2} 
	- (m\nu+2-n) \nu b r^{-m\nu-2} )\\
	&+ ( r^{-m\lambda-2} -b r^{-m\nu-2} ) \Delta_\omega(\alpha^m).
\end{aligned}
\]
By the definition of $\rho$, we have 
\[
\begin{aligned}
	(w_{in}^-)_t - \Delta (w_{in}^-)^m &=
	- \frac{1}{m} B b 
	\alpha r^{-\lambda+m(\lambda-\nu)} 
	( 1-b r^{m(\lambda-\nu)} )^{\frac{1}{m}-1}
	- ( r^{-m\lambda-2} -b r^{-m\nu-2} ) \Delta_\omega(\alpha^m)  \\
	&\quad 
	- m \alpha^m ( (m\lambda+2-n) \lambda r^{-m\lambda-2} 
	- (m\nu+2-n) \nu b r^{-m\nu-2} )  \\
	&\leq 
	- \frac{1}{m} B b 
	\alpha r^{-\lambda+m(\lambda-\nu)} ( 1-b \rho^{m(\lambda-\nu)} )^{\frac{1}{m}-1}
	+ C r^{-m\lambda-2} + C b r^{-m\nu-2}  \\
	&=
	- \frac{1}{m} B b 
	\alpha r^{-\lambda+m(\lambda-\nu)} \delta^{\frac{1}{m}-1}
	+ C r^{-m\lambda-2} + C b r^{-m\nu-2} 
\end{aligned}
\]
for $r\leq \rho$, where $C>0$ is a constant independent of $b$. 
By $\alpha_{min}>0$, $\lambda>\nu$,
$b>1$, and $r\leq \rho<1$, we have 
\[
\begin{aligned}
	(w_{in}^-)_t - \Delta (w_{in}^-)^m &\leq 
	- \frac{1}{m} B b \delta^{\frac{1}{m}-1}
	\alpha_{min} r^{-\lambda+m(\lambda-\nu)} + C b r^{-m\lambda-2}  \\
	&=
	- b r^{ -\lambda+m(\lambda-\nu) } 
	\left(  \frac{1}{m} \delta^{\frac{1}{m}-1} \alpha_{min} B 
	- C r^{(1-m)\lambda-2-m(\lambda-\nu)} \right). 
\end{aligned}
\]
Recall that $(1-m)\lambda-2-m(\lambda-\nu)>0$. 
Thus, 
\[
\begin{aligned}
	(w_{in}^-)_t - \Delta (w_{in}^-)^m
	\leq 
	- b r^{ -\lambda+m(\lambda-\nu) } 
	\left(  \frac{1}{m} \delta^{\frac{1}{m}-1} \alpha_{min} B - C\right)
\end{aligned}
\]
for $r\leq \rho$. 
Hence, by choosing $B>1$ large, we conclude that 
$w_{in}^-$ is a subsolution for $r\leq \rho(t)$. 

\textbf{Matching condition:}
Since both $w_{in}^-$ and $w_{out}^-$ are of the separated form $\alpha(\omega) f(r,t)$, it is sufficient to check that 
\[
	\left. \frac{\partial}{\partial r} (w_{in}^-)^m  \right|_{r=\rho(t)}
	<
	\left. \frac{\partial}{\partial r} (w_{out}^-)^m  \right|_{r=\rho(t)}. 
\]
By the definition of $\rho$ and the choice of $\delta$, we have 
\[
\begin{aligned}
	\left. \frac{\partial}{\partial r} (w_{out}^-)^m  \right|_{r=\rho(t)}
	- 
	\left. \frac{\partial}{\partial r} (w_{in}^-)^m  \right|_{r=\rho(t)}
	&= 
	-\alpha^m m\mu \delta \rho^{-m\lambda-1}
	- \alpha^m (-m\lambda \rho^{-m\lambda-1} + m\nu b\rho^{-m\nu-1}) \\
	&= 
	\alpha^m m \rho^{-m\lambda-1}
	\left( 
	- \mu \delta + \lambda - \nu b\rho^{m(\lambda-\nu)} 
	\right) \\
	&=
	\alpha^m m \rho^{-m\lambda-1}
	\big( - \mu \delta + \lambda - (1-\delta) \nu \big) \\
	&= 
	\alpha^m m \rho^{-m\lambda-1}
	\big( \lambda-\nu-(\mu-\nu)\delta \big) > 0. 
\end{aligned}
\]

\textbf{Outer part:}
Note that 
\[
	\rho'(t)=
	-\frac{1}{m(\lambda-\nu)} B\rho(t) \leq0. 
\]
From this, it follows that 
\[
	(w_{out}^-)_t =
	\alpha \delta^\frac{1}{m} (\mu-\lambda) \rho^{\mu-\lambda-1} \rho' r^{-\mu} 
	\leq 0. 
\]
By direct computations, we have 
\[
\begin{aligned}
	\Delta (w_{out}^-)^m &= 
	\alpha^m \delta \rho^{m(\mu-\lambda)} 
	m\mu (m\mu+2-n) r^{-m\mu-2} 
	+ \delta \rho^{m(\mu-\lambda)} r^{-m\mu-2} \Delta_\omega(\alpha^m).
\end{aligned}
\]
Then~\eqref{eq:asmu} implies  
\[
\begin{aligned}
	(w_{out}^-)_t - \Delta (w_{out}^-)^m &\leq 
	- \alpha^m \delta \rho^{m(\mu-\lambda)} 
	m\mu (m\mu+2-n) r^{-m\mu-2} 
	- \delta \rho^{m(\mu-\lambda)} r^{-m\mu-2} \Delta_\omega(\alpha^m)  \\
	&= 
	-\delta \rho^{m(\mu-\lambda)} r^{-m\mu-2}
	\left( \alpha^m m\mu (m\mu+2-n) 
	+\Delta_\omega(\alpha^m)  \right) \\
	&\leq 
	-\delta \rho^{m(\mu-\lambda)} r^{-m\mu-2}
	\left( \alpha_{min}^m m\mu (m\mu+2-n) 
	- \max_{\omega\in S^{n-1}}|\Delta_\omega(\alpha^m)|  \right) 
	\leq 0. 
\end{aligned}
\]
Hence $w_{out}^-$ is a subsolution for $r\geq \rho(t)$. 
\end{proof}

\subsection{Completion of the proof of Theorem~\ref{th:anisotropic}} 
\label{subsec:compl}

\begin{prop} \label{pro:n}
    Let $n \geq 2$ and assume \emph{(A1)}, \emph{(A2)}, and \emph{(A3)}.
    Then there exist a supersolution $w^+$ and a subsolution $w^-$ of~\eqref{eq:wprob}, which have for each $\omega \in S^{n-1}$ and $t\geq 0$ the asymptotic behavior
    \begin{equation*}
        w^+(r,\omega,t)^m = \alpha^m(\omega) r^{-m\lambda} + O(r^{-m\nu}), \quad
        w^-(r,\omega,t)^m = \alpha^m(\omega) r^{-m\lambda} + O(r^{-m\nu})
    \end{equation*}
    as $r\to 0$. Moreover, 
    \begin{equation*}
        w^-(r,\omega,t) \leq w^+(r,\omega,t) 
    \end{equation*}
    and
    \begin{equation*}
        w^-(r,\omega,0) \leq u_0(x) \leq w^+(r,\omega,0)
    \end{equation*}
    for all $r>0$, $\omega \in S^{n-1}$, which are defined in~\eqref{eq:r_omega}, and $t\geq 0$. 
\end{prop}

\begin{proof}
We choose $w^+$ and $w^-$ as in Lemmata~\ref{l+:n} and~\ref{l-:n}, respectively. 
Note that 
\[
\begin{aligned}
	&w^+(r,\omega,0)= \big( \alpha^m(\omega) r^{-m\lambda} 
	+ A r^{-m\nu}+A\big)^\frac{1}{m}, \\
	&	w^-(r,\omega,0)= 
	\left\{ 
	\begin{aligned} 
	& \alpha(\omega) r^{-\lambda} 
	( 1-b_0 r^{m(\lambda-\nu)} )^\frac{1}{m} 
	&&\mbox{ for }r\leq \rho(0),  \\
	& \alpha(\omega) \delta^\frac{1}{m} 
	\rho^{\mu-\lambda}(0) r^{-\mu}
	&&\mbox{ for }r> \rho(0). 
	\end{aligned}
	\right.
\end{aligned}
\]
Moreover, $\rho(0)<1$. 
Then by choosing $A$ and $b_0$ sufficiently large 
and $\delta$ sufficiently small, 
we see that the function $u_0$ satisfying (A3) 
can be always squeezed in between comparison functions 
$w^-(r,\omega,0)$ and $w^+(r,\omega,0)$.
\end{proof}

\begin{proof} [Proof of Theorem~\ref{th:anisotropic}]
In Proposition~\ref{pro:n} we proved the existence of a global-in-time sub- and supersolution of problem~\eqref{eq:wprob}, which implies the existence of sub- and supersolution of~\eqref{eq:main} with the desired asymptotic behavior. 
The rest of the proof of Theorem~\ref{th:anisotropic} is the same as in Section~5 in~\cite{FTY}. Here, it was proved that the existence of global-in-time comparison functions, i.e. sub- and supersolution of~\eqref{eq:main}, which are positive and bounded on each compact subset of $(\R^{n}\setminus\{\xi_0\})\times(0,\infty)$, implies the existence of a global-in-time solution of~\eqref{eq:main}-\eqref{eq:main_initial}.
\end{proof}

\section{Proof of Theorem~\ref{th:dirac}}
\label{subsec:proof_th3}

\begin{proof} [Proof of Theorem~\ref{th:dirac}]
For simplicity, let $B_R:=B_R \left(\xi(t)\right)$ denote an open ball in $\R^n$ of radius $R$ centered at $\xi(t)$.
For $\varepsilon > 0$ let $\eta_\varepsilon \in C^2(\R)$ be a non-negative cut-off function such that $\eta_\varepsilon(r) \equiv 0$ for $r\leq \varepsilon$, 
$\eta_\varepsilon(r) \equiv 1$ for $r\geq 3\varepsilon$, 
$\eta_\varepsilon''\geq 0$ for $r \in [\varepsilon, 2\varepsilon]$, 
$\eta_\varepsilon''\leq 0$ for $r \in [2\varepsilon, 3\varepsilon]$, 
and $0 \leq \eta_\varepsilon'(r) \leq \eta_\varepsilon'(2\varepsilon) = \tilde c_1 \varepsilon^{-1}$ for some $\tilde c_1 >0$ and $|\eta_\varepsilon''| \leq \tilde c_2 \varepsilon^{-2}$ for some $\tilde c_2>0$.

Let $u$ be from Theorem~1.1 in~\cite{FTY}, that means that $u$ is a classical solution of~\eqref{eq:porous}-\eqref{eq:porous-init} such that $u \in C([0,\infty);L^1_{loc}(\R^n))$. Let $\varphi \in C_0^\infty (\R^n \times (0, \infty))$ and set $\varphi_\varepsilon(x,t) := \eta_\varepsilon(|x-\xi(t)|) \varphi(x,t)$.
Without loss of generality, we may assume that there is a nonempty open time interval $I \subset (0,\infty)$ such that $\xi(t) \in \supp \varphi(\cdot,t)$ for all $t \in I$. We can fix $\varepsilon$ sufficiently small so that $B_{3\varepsilon} \subset \supp \varphi(\cdot,t)$ for all $t \in I$.
Multiplying now equation~\eqref{eq:porous} by $\varphi_\varepsilon$ and integrating it over $\R^n \times (0, \infty)$, we obtain
\begin{equation} \label{eq:int_in_y}
    \int_0^\infty \int_{\R^n} u_t \varphi_\varepsilon \, dx \, dt = \int_0^\infty \int_{\R^n} \Delta u^m \varphi_\varepsilon \, dx \, dt.
\end{equation}
Let us denote
\begin{align*} 
\begin{split}
    I_\varepsilon &:=  \int_{B_{3\varepsilon}\setminus B_{\varepsilon}} u^m \eta_\varepsilon\Delta\varphi \, dx, \quad
    J_\varepsilon := 2\int_{B_{3\varepsilon}\setminus B_{\varepsilon}} u^m \nabla\eta_\varepsilon \cdot \nabla\varphi \, dx,  \\
    K_\varepsilon &:= \int_{B_{3\varepsilon}\setminus B_{\varepsilon}} u^m \varphi\Delta\eta_\varepsilon \, dx,\quad
    H_\varepsilon := \int_{B_{3\varepsilon}\setminus B_\varepsilon} u \eta_\varepsilon \varphi_t \, dx .
\end{split}
\end{align*}
Since $\varphi$ is smooth and compactly supported in $\R^n \times (0, \infty)$, integrating the right-hand side of~\eqref{eq:int_in_y} by parts we have
\begin{equation}\label{eq:est0}
    \int_0^\infty \int_{\R^n} \Delta u^m \varphi_\varepsilon \, dx \, dt
    = \int_0^\infty \int_{\R^n \setminus B_{3\varepsilon}} u^m \Delta\varphi \, dx\, dt + \int_0^\infty \left( I_\varepsilon + J_\varepsilon + K_\varepsilon \right)\, dt .
\end{equation}
Similarly, we analyze the left-hand side of~\eqref{eq:int_in_y} and obtain
\begin{equation*}
    \int_0^\infty \int_{\R^n} u_t \varphi_\varepsilon \, dx \, dt = - \int_0^\infty \int_{\R^n} u (\varphi_\varepsilon)_t \, dx \, dt = 
    - \int_0^\infty \int_{\R^n \setminus B_{3\varepsilon}} u \varphi_t \, dx \, dt - \int_0^\infty H_\varepsilon \, dt.
\end{equation*}
Hence, equation~\eqref{eq:int_in_y} can be written as
\begin{equation*}
    -\int_0^\infty \int_{\R^n \setminus B_{3\varepsilon}} \big( u \varphi_t + u^m \Delta\varphi \big) \, dx \, dt 
    = \int_0^\infty \left(H_\varepsilon + I_\varepsilon + J_\varepsilon + K_\varepsilon \right) \, dt.
\end{equation*}
In the following, we show that 
\begin{equation} \label{eq:assertion}
    H_\varepsilon + I_\varepsilon + J_\varepsilon + K_\varepsilon \to (n-2) |S^{n-1}| k^m(t) \varphi(\xi(t),t) 
\end{equation}
locally uniformly for $t$ as $\varepsilon \to 0$.
In order to do that, we choose $\varepsilon$ sufficiently small so that the method of sub- and supersolutions in~\cite{FTY} provides estimates of the form
\begin{equation} \label{eq:v_estimates}
\begin{split}
    u^m(x,t) &\leq k^m(t) \left(|x-\xi(t)|^{2-n} + b(t)|x-\xi(t)|^{-\lambda}\right),\\
    u^m(x,t) &\geq k^m(t) \left(|x-\xi(t)|^{2-n} - b(t)|x-\xi(t)|^{-\lambda}\right)_+,
\end{split}
\end{equation}
for all $(x,t) \in B_{3\varepsilon} \times I$. 
Here, $b(t)=b_0e^{B t}$ for some constants $B$, $b_0 >1$ and $\lambda < n-2$.
First, we deal with the integrals $H_\varepsilon, I_\varepsilon, J_\varepsilon$ and show that they converge to zero locally uniformly for $t$.
In what follows, by $c$ we will denote a large enough but otherwise arbitrary constant independent of $t$ and $\varepsilon$.
Given that $|\eta_\varepsilon| \leq 1$, for $m>m_c$ and $t<T$ for some $T>0$ it holds that
\[
\begin{aligned}
    |H_\varepsilon| &= \left|\int_{B_{3\varepsilon}\setminus B_\varepsilon} u \eta_\varepsilon \varphi_t \, dx \right| 
    \leq \sup_{B_{3\varepsilon}\setminus B_{\varepsilon}}|\varphi_t| \int_{B_{3\varepsilon}\setminus B_\varepsilon} u \, dx
    \leq c \int_{\varepsilon}^{3\varepsilon} \big(r^{2-n} + b(t) r^{-\lambda} \big)^{\frac{1}{m}} r^{n-1} \, dr \leq \\
    &\leq c \int_{\varepsilon}^{3\varepsilon} r^{\frac{2-n}{m} + n-1} \to 0 \quad \text{ as } \quad \varepsilon \to 0.
\end{aligned}
\]
Similarly, 
\[
\begin{aligned}
    |I_\varepsilon| &= \left| \int_{B_{3\varepsilon}\setminus B_{\varepsilon}} u^m \eta_\varepsilon\Delta\varphi \, dx \right| 
    \leq \sup_{B_{3\varepsilon}\setminus B_{\varepsilon}}|\Delta\varphi| \int_{B_{3\varepsilon}\setminus B_\varepsilon} u^m \, dx
    \leq c \int_{\varepsilon}^{3\varepsilon} \big(r^{2-n} + b(t) r^{-\lambda} \big) r^{n-1} \, dr \leq \\
    &\leq c \int_{\varepsilon}^{3\varepsilon} r \, dr \to 0 \quad \text{ as } \quad \varepsilon \to 0.
\end{aligned}
\]
Moreover, using $0 \leq\eta_\varepsilon' \leq \tilde c_1 \varepsilon^{-1}$, we obtain
\[
\begin{aligned}
    |J_\varepsilon| &= \left| 2 \int_{B_{3\varepsilon}\setminus B_{\varepsilon}} u^m \nabla\eta_\varepsilon \cdot \nabla\varphi \, dx \right| 
    \leq 2 \tilde c_1 \varepsilon^{-1} \sup_{B_{3\varepsilon}\setminus B_{\varepsilon}} |\omega \cdot\nabla\varphi| \int_{B_{3\varepsilon}\setminus B_\varepsilon} u^m \, dx\leq \\
    &\leq c \, \varepsilon^{-1} \int_{\varepsilon}^{3\varepsilon} \big(r^{2-n} + b(t) r^{-\lambda} \big) r^{n-1} \, dr 
    \leq c \varepsilon^{-1} \int_{\varepsilon}^{3\varepsilon} r \, dr \to 0 
    \quad \text{ as } \quad\varepsilon \to 0.
\end{aligned}
\]
Now we deal with the integral $K_\varepsilon$. Denoting
\[
   K_\varepsilon^1 := \int_{B_{3\varepsilon}\setminus B_{\varepsilon}} u^m \varphi |x-\xi(t)|^{-1}\eta_\varepsilon' \, dx, \quad
   K_\varepsilon^2 := \int_{B_{2\varepsilon}\setminus B_{\varepsilon}} u^m \varphi \eta_\varepsilon'' \, dx, \quad
   K_\varepsilon^3 := \int_{B_{3\varepsilon}\setminus B_{2\varepsilon}} u^m \varphi (-\eta_\varepsilon)'' \, dx,
\]
we can split $K_\varepsilon$ into
\begin{equation*}
    K_\varepsilon = (n-1) K_\varepsilon^1 + K_\varepsilon^2 - K_\varepsilon^3.
\end{equation*}
By means of the non-negativity of $u^m$ and the properties $\eta_\varepsilon' \geq 0$, $\eta_\varepsilon'' \geq 0$ on $B_{2\varepsilon}\setminus B_{\varepsilon}$, and $\eta_\varepsilon'' \leq 0$ on $B_{3\varepsilon}\setminus B_{2\varepsilon}$, 
we obtain
\[
\begin{aligned}
   &\inf_{B_{3\varepsilon}\setminus B_{\varepsilon}} \varphi \int_{B_{3\varepsilon}\setminus B_{\varepsilon}} u^m |x-\xi(t)|^{-1}\eta_\varepsilon' \, dx
   \leq K_\varepsilon^1 \leq 
   \sup_{B_{3\varepsilon}\setminus B_{\varepsilon}} \varphi \int_{B_{3\varepsilon}\setminus B_{\varepsilon}} u^m |x-\xi(t)|^{-1}\eta_\varepsilon' \, dx, \\
   &\inf_{B_{2\varepsilon}\setminus B_{\varepsilon}} \varphi \int_{B_{2\varepsilon}\setminus B_{\varepsilon}} u^m \eta_\varepsilon'' \, dx
   \leq K_\varepsilon^2 \leq 
   \sup_{B_{2\varepsilon}\setminus B_{\varepsilon}} \varphi \int_{B_{2\varepsilon}\setminus B_{\varepsilon}} u^m \eta_\varepsilon'' \, dx, \\
	&\inf_{B_{3\varepsilon}\setminus B_{2\varepsilon}} \varphi \int_{B_{3\varepsilon}\setminus B_{2\varepsilon}} u^m (-\eta_\varepsilon)'' \, dx
   \leq K_\varepsilon^3 \leq 
   \sup_{B_{3\varepsilon}\setminus B_{2\varepsilon}} \varphi \int_{B_{3\varepsilon}\setminus B_{2\varepsilon}} u^m (-\eta_\varepsilon)'' \, dx.
\end{aligned}
\]
We only consider the case where 
$\sup_{B_{3\varepsilon}\setminus B_{\varepsilon}} \varphi$ and $\inf_{B_{3\varepsilon}\setminus B_{\varepsilon}} \varphi$ are non-negative. 
The other cases can be handled in the same way by using ~\eqref{eq:est_L} below.
Indeed, a change of the sign of $\sup_{B_{3\varepsilon}\setminus B_{\varepsilon}} \varphi$ and $\inf_{B_{3\varepsilon}\setminus B_{\varepsilon}} \varphi$ 
will not change the limit value. 

For $a:= n-2-\lambda>0$ we set
\begin{equation*}
    L_\varepsilon^1 := b(t) \int_\varepsilon^{2\varepsilon} r^{a+1} \eta_\varepsilon'' \, dr, \quad
    L_\varepsilon^2 := b(t) \int_{2\varepsilon}^{3\varepsilon} r^{a+1} (-\eta_\varepsilon)'' \, dr, \quad
    L_\varepsilon^3 := b(t) \int_\varepsilon^{3\varepsilon} r^a \eta_\varepsilon' \, dr.
\end{equation*}
Since $0 \leq \eta_\varepsilon' \leq \tilde c_1 \varepsilon^{-1}$, for $t<T$ with some $T>0$ it holds that
\begin{equation*} 
     | L_\varepsilon^3 | = \left| b(t)\int_\varepsilon^{3\varepsilon} r^a \eta_\varepsilon' \, dr \right| 
    \leq c \varepsilon^a,
\end{equation*}
and by $|\eta_\varepsilon''| \leq \tilde c_2 \varepsilon^{-2}$ we have
\begin{equation*} 
    | L_\varepsilon^1 | + | L_\varepsilon^2 | = \left| b(t) \int_\varepsilon^{2\varepsilon} r^{a+1} \eta_\varepsilon'' \, dr \right|
    +
    \left| b(t)\int_{2\varepsilon}^{3\varepsilon} r^{a+1} (-\eta_\varepsilon)'' \, dr \right| 
    \leq c \varepsilon^a.
\end{equation*}
Hence, 
\begin{equation} \label{eq:est_L}
    L_\varepsilon^1 \to 0, \quad  L_\varepsilon^2 \to 0, \quad L_\varepsilon^3 \to 0 \quad \text{ as } \quad \varepsilon \to 0
\end{equation}
locally uniformly for $t$.
Using inequalities~\eqref{eq:v_estimates}, we have 
\begin{equation} \label{eq:k1_up}
    K_\varepsilon^1
    \leq |S^{n-1}| k^m(t) \sup_{B_{3\varepsilon}\setminus B_{\varepsilon}} \varphi  \int_\varepsilon^{3\varepsilon} \big(1 + b(t)r^{a}\big)\eta_\varepsilon' \, dr
    = |S^{n-1}| k^m(t) \sup_{B_{3\varepsilon}\setminus B_{\varepsilon}} \varphi \left( 1 + L_\varepsilon^3 \right),
\end{equation}
and
\begin{equation} \label{eq:k1_down}
    K_\varepsilon^1
    \geq |S^{n-1}| k^m(t) \inf_{B_{3\varepsilon}\setminus B_{\varepsilon}} \varphi \int_\varepsilon^{3\varepsilon} \big(1 - b(t)r^{a}\big) \eta_\varepsilon' \, dr 
    = |S^{n-1}| k^m(t) \inf_{B_{3\varepsilon}\setminus B_{\varepsilon}} \varphi \left( 1 - L_\varepsilon^3 \right).
\end{equation}
Integrating by parts and estimating integrals $K_\varepsilon^2$ and $K_\varepsilon^3$ we have
\begin{equation*}
\begin{split}
    K_\varepsilon^2
    &\leq |S^{n-1}| k^m(t) \sup_{B_{2\varepsilon}\setminus B_{\varepsilon}} \varphi 
    \int_\varepsilon^{2\varepsilon} \big(r+b(t)r^{a+1}\big) \eta_\varepsilon'' \, dr =\\
    &= |S^{n-1}| k^m(t) \sup_{B_{2\varepsilon}\setminus B_{\varepsilon}} \varphi 
    \left( 2\varepsilon \eta_\varepsilon'(2\varepsilon) - \eta_\varepsilon(2\varepsilon) + L_\varepsilon^1 \right),
\end{split}
\end{equation*}
and
\begin{equation*}
\begin{split}
    K_\varepsilon^3 
    &\geq |S^{n-1}| k^m(t) \inf_{B_{3\varepsilon}\setminus B_{2\varepsilon}} \varphi 
    \int_{2\varepsilon}^{3\varepsilon} \big(r-b(t)r^{a+1}\big) (-\eta_\varepsilon)'' \, dr =\\
    &= |S^{n-1}| k^m(t) \inf_{B_{3\varepsilon}\setminus B_{2\varepsilon}} \varphi 
    \left( 2\varepsilon \eta_\varepsilon'(2\varepsilon) + 1 - \eta_\varepsilon(2\varepsilon) - L_\varepsilon^2 \right).
\end{split}
\end{equation*}
Thus, 
\begin{equation} \label{eq:k23_up}
\begin{split}
    K_\varepsilon^2 - K_\varepsilon^3 
    \leq  \, - |S^{n-1}| k^m(t) 
    \Big[ &(1- L_\varepsilon^2) \inf_{B_{3\varepsilon}\setminus B_{2\varepsilon}}\varphi 
    - L_\varepsilon^1 \sup_{B_{2\varepsilon}\setminus B_{\varepsilon}} \varphi \, - \\
    &  - \Big( \sup_{B_{2\varepsilon}\setminus B_{\varepsilon}}\varphi - \inf_{B_{3\varepsilon}\setminus B_{2\varepsilon}}\varphi \Big) \big( 2\varepsilon\eta_\varepsilon'(2\varepsilon) - \eta_\varepsilon(2\varepsilon) \big) \Big].
\end{split}
\end{equation}
Analogously, 
\begin{equation*}
\begin{split}
    K_\varepsilon^2
    &\geq |S^{n-1}| k^m(t) \inf_{B_{2\varepsilon}\setminus B_{\varepsilon}} \varphi 
    \int_\varepsilon^{2\varepsilon} \big(r-b(t)r^{a+1}\big) \eta_\varepsilon'' \, dr = \\
    &= |S^{n-1}| k^m(t) \inf_{B_{2\varepsilon}\setminus B_{\varepsilon}} \varphi 
    \left( 2\varepsilon \eta_\varepsilon'(2\varepsilon) - \eta_\varepsilon(2\varepsilon) - L_\varepsilon^1 \right),
\end{split}
\end{equation*}
and
\begin{equation*}
\begin{split}
    K_\varepsilon^3 & 
    \leq |S^{n-1}| k^m(t) \sup_{B_{3\varepsilon}\setminus B_{2\varepsilon}} \varphi 
    \int_{2\varepsilon}^{3\varepsilon} \big(r+b(t)r^{a+1}\big) (-\eta_\varepsilon)'' \, dr = \\
    &= |S^{n-1}| k^m(t) \sup_{B_{3\varepsilon}\setminus B_{2\varepsilon}} \varphi 
    \left( 2\varepsilon \eta_\varepsilon'(2\varepsilon) + 1 - \eta_\varepsilon(2\varepsilon) + L_\varepsilon^2 \right).
\end{split}
\end{equation*}
Hence,
\begin{equation} \label{eq:k23_down}
\begin{split}
    K_\varepsilon^2 - K_\varepsilon^3 
    \geq  \, - |S^{n-1}| k^m(t) \Big[ &(1 + L_\varepsilon^2) \sup_{B_{3\varepsilon}\setminus B_{2\varepsilon}}\varphi +  L_\varepsilon^1 \inf_{B_{2\varepsilon}\setminus B_{\varepsilon}} \varphi \, - \\
    &- \Big( \inf_{B_{2\varepsilon}\setminus B_{\varepsilon}}\varphi - \sup_{B_{3\varepsilon}\setminus B_{2\varepsilon}}\varphi \Big) \big( 2\varepsilon\eta_\varepsilon'(2\varepsilon) - \eta_\varepsilon(2\varepsilon) \big) \Big] .
\end{split}
\end{equation}
Finally, using~\eqref{eq:est_L}, inequalities~\eqref{eq:k1_up},~\eqref{eq:k1_down},~\eqref{eq:k23_up}, and~\eqref{eq:k23_down} yield
\begin{equation*}
    K_\varepsilon \to (n-2) |S^{n-1}| k^m(t) \varphi(\xi(t),t)  \quad \text{ as } \quad\varepsilon \to 0
\end{equation*}
locally uniformly for $t$.
Thus, the assertion~\eqref{eq:assertion} is proved, which completes the proof.
\end{proof}

\section{Proof of Theorem~\ref{th:anisotropic_weak}} 
\label{subsec:proof_anisotropic_weak}

\begin{proof} [Proof of Theorem~\ref{th:anisotropic_weak}~\ref{it:i}]
The functions from Theorem~\ref{th:anisotropic} satisfy $u \in C([0,\infty);L^1_{loc}(\R^n))$. The proof of this statement is analogous to the proof of Lemma~5.2 in~\cite{FTY}, and so we omit it here.

Since in this case $\lambda < (n-2)/m$ by $\lambda<n$, it necessarily holds that the singularity of solutions from Theorem~\ref{th:anisotropic} is weaker than the singularity of solutions of the type~\eqref{eq:steady_state} and~\eqref{eq:rad_as_sym}. It suggests that in the distributional sense, unlike solutions of the type~\eqref{eq:steady_state} and~\eqref{eq:rad_as_sym} satisfy equations~\eqref{eq:fund_weak} and~\eqref{eq:n3_weak_delta} with a singular source term, solutions from Theorem~\ref{th:anisotropic_weak}~\ref{it:i} satisfy equation~\eqref{eq:weak} with no source term on the right-hand side.
Rigorously, the proof of this theorem can be carried out analogously to the proof of Theorem~\ref{th:dirac} in Section~\ref{subsec:proof_th3} (and it is less technical due to standing versus moving singularity).
\end{proof}

\begin{proof} [Proof of Theorem~\ref{th:anisotropic_weak}~\ref{it:ii}]
To prove that $u \notin L^p_{loc}(\R^n \times [0,\infty))$ for any $p \geq 1$ if $\lambda \geq n$, we integrate over $B_1(\xi_0) \times [0,1]$, and use the comparison function $w^-$ to estimate the integral
\begin{equation*}
    I:=\int_0^1 \int_{B_1(\xi_0)} u^p(x,t) \, dx \, dt \geq \int_0^1 \int_{B_1(\xi_0)} (w^-(r,\omega,t))^p \, dx \, dt.
\end{equation*}
By the definition of $w^- \geq 0$ in~\eqref{eq:subsolform}, $\rho(t)<1$, and $\alpha \geq \alpha_{min} >0$, we have
\begin{equation*}
    I 
    \geq \int_0^1 \int_{B_{\rho(t)}(\xi_0)} (w^-(r,\omega,t))^p \, dx \, dt
    \geq \alpha_{min}^p \int_0^1 \int_0^{\rho(t)} r^{n-1-p\lambda} \big( 1-b(t) r^{m(\lambda-\nu)} \big)^\frac{p}{m} \, dr \, dt.
\end{equation*}
Substituting $z=b(t) r^{m(\lambda-\nu)}$, by $b \geq 1$ we obtain
\[
\begin{aligned}
    I 
    &\geq \frac{\alpha_{min}^p}{m(\lambda-\nu)} \int_0^1 b^{\frac{p\lambda-n}{m(\lambda-\nu)}}(t) \int_0^{1-\delta} z^{-1-\frac{p\lambda-n}{m(\lambda-\nu)}} (1-z)^\frac{p}{m} \, dz \, dt \\
    &\geq \frac{\alpha_{min}^p \delta^\frac{p}{m}}{m(\lambda-\nu)} \int_0^1 \int_0^{1-\delta} z^{-1-\frac{p\lambda-n}{m(\lambda-\nu)}} \, dz \, dt.
\end{aligned}
\]
This integral is infinite exactly when $\lambda \geq n$ for any $p \geq 1$, which completes the proof.
\end{proof}

\section{Non-integrability of the singular traveling wave} 
\label{subsec:proof_snaking}

In this section we show that for $U$ from~\eqref{eq:snaking} it holds that $U \notin L^p_{loc}(\R^n \times \R)$ for any $p \geq 1$. Without loss of generality, we may take $a = e_n = (0,\ldots,0,1)$. Indeed, given any velocity vector $a \in \R^n$, we could transform the coordinate system and proceed as below. Let $B'_1:= \{ x' \in \R^{n-1}; |x'|<1\}$.
For $p \geq 1$ we examine the integrability over $[0,1] \times B'_1 \times [-1,0]$, i.e.
\begin{equation*}
    I:= \int_0^1 \int_{B'_1 \times [-1,0]} U^p(x,t) \, dx \, dt 
    = \int_0^1 \int_{-1}^0 \int_{B'_1} C^p \big( \sqrt{|x'|^2+(x_n-t)^2} + (x_n-t) \big)^{-\frac{p}{1-m}} \, dx' \, dx_n \, dt.
\end{equation*}
By the change of variables $y_n=-(x_n-t)$ and $|x'| = y_n r$, we obtain
\begin{equation*}
\begin{split}
    I
    &= C^p |S^{n-2}| \int_0^1 \int_{1+t}^t \int_0^{1/y_n} (y_n r)^{n-2} \big( \sqrt{y_n^2 r^2 + y_n^2} - y_n \big)^{-\frac{p}{1-m}} y_n \, dr \, (-dy_n) \, dt, \\
    &= C^p |S^{n-2}| \int_0^1 \int_t^{1+t} y_n^{n-1-\frac{p}{1-m}} \int_0^{1/y_n} r^{n-2} \big( \sqrt{r^2 + 1} - 1 \big)^{-\frac{p}{1-m}} \, dr \, dy_n \, dt,\\
    &= C^p |S^{n-2}| \int_0^1 \int_t^{1+t} y_n^{n-1-\frac{p}{1-m}} \int_0^{1/y_n} r^{n-2-\frac{2p}{1-m}} \big( \sqrt{r^2 + 1} + 1 \big)^{\frac{p}{1-m}} \, dr \, dy_n \, dt.
\end{split}
\end{equation*}
We have $1/y_n \geq 1/2$, hence
\begin{equation*}
    I \geq C^p \int_0^1 \int_1^{1+t} y_n^{n-1-\frac{p}{1-m}} \int_0^{1/2} r^{n-2-\frac{2p}{1-m}} \, dr \, dy_n \, dt,
\end{equation*}
which is finite if $p < (1-m)(n-1)/2$. 
Since we assumed that $p \geq 1$, $m>(n-3)/(n-1)=m^*$ for $n\geq3$ 
and $m>0=m^*$ for $n=2$, this condition for $p$ cannot be satisfied. This implies the conclusion.

\subsection*{Acknowledgment}
We thank the referee for the comments that helped improve the presentation significantly.

\end{document}